\documentclass[a4paper,11pt]{amsart}

\usepackage[french, ngerman, italian, english]{babel}
\usepackage{amsmath,amsthm, amssymb}
\usepackage{setspace}
\usepackage{anysize}
\usepackage{url}
\usepackage{bm}
\usepackage{graphicx}
\usepackage[colorlinks=true]{hyperref}
\usepackage{mathptmx}
\usepackage{paralist}
\usepackage{verbatim}
\usepackage[utf8]{inputenc}

\theoremstyle{plain}
\newtheorem{thm}{\bf Theorem}[section]

\newtheorem{prop}[thm]{\bf Proposition}
\newtheorem{lemma}[thm]{\bf Lemma}
\newtheorem{corollary}[thm]{\bf Corollary}

\theoremstyle{definition}
\newtheorem{definition}[thm]{\bf Definition}
\newtheorem{discussion}[thm]{\bf Discussion}

\theoremstyle{remark}
\newtheorem{remark}[thm]{\bf Remark}
\newtheorem{example}[thm]{\bf Example}
\theoremstyle{example}

\def\NN{{\NZQ N}}

\def\ZZ{{\NZQ Z}}

\def\frk{\frak}               

\def\Phi{{\frk n}}
\def\Phi{{\frk N}}
\def\opn#1#2{\def#1{\operatorname{#2}}} 
\opn\chara{char} \opn\length{\ell} \opn\pd{pd} \opn\rk{rk}
\opn\projdim{proj\,dim} \opn\injdim{inj\,dim} \opn\rank{rank}
\opn\depth{depth} \opn\grade{grade} \opn\height{height}
\opn\embdim{emb\,dim} \opn\codim{codim}

\opn\Tr{Tr} \opn\bigrank{big\,rank}
\opn\superheight{superheight}\opn\lcm{lcm}
\opn\trdeg{tr\,deg}
\opn\reg{reg} \opn\lreg{lreg} \opn\ini{in} \opn\lpd{lpd}
\opn\size{size} \opn\sdepth{sdepth}
\opn\link{link}\opn\fdepth{fdepth}\opn\lex{lex}
\opn\div{div} \opn\Div{Div} \opn\cl{cl} \opn\Cl{Cl}
\opn\Spec{Spec} \opn\Supp{Supp} \opn\supp{supp} \opn\Sing{Sing}
\opn\Ass{Ass} \opn\Min{Min}\opn\Mon{Mon}
\opn\Ann{Ann} \opn\Rad{Rad} \opn\Soc{Soc}
\opn\Im{Im} \opn\Ker{Ker} \opn\Coker{Coker} \opn\Am{Am}
\opn\Hom{Hom} \opn\Tor{Tor} \opn\Ext{Ext} \opn\End{End}
\opn\Aut{Aut} \opn\id{id}

\opn\nat{nat}
\opn\pff{pf}
\opn\Pf{Pf} \opn\GL{GL} \opn\SL{SL} \opn\mod{mod} \opn\ord{ord}
\opn\Gin{Gin} \opn\Hilb{Hilb}\opn\sort{sort}
\opn\aff{aff} \opn\con{conv} \opn\relint{relint} \opn\st{st}
\opn\lk{lk} \opn\cn{cn} \opn\core{core} \opn\vol{vol}
\opn\link{link} \opn\star{star}\opn\lex{lex}\opn\set{set}
\opn\gr{gr}

\def\pot#1#2{#1[\kern-0.28ex[#2]\kern-0.28ex]}

\opn\dirlim{\underrightarrow{\lim}}
\opn\inivlim{\underleftarrow{\lim}}
\def \GL{{\operatorname{GL}}}
\def \HF{{\operatorname{HF}}}

\def \chara{{\operatorname{char}}}

\def \height{{\operatorname{ht}}}
\def \reg{{\operatorname{reg}}}
\def \depth{{\operatorname{depth}}}
\def \Gin{{\operatorname{Gin}}}
\def \grade{{\operatorname{grade}}}
\def \Ker{{\operatorname{Ker}}}
\def \pd{{\operatorname{pd}}}

\def \NN{\mathbb N}

\def \ZZ{\mathbb Z}

\def \S{\mathcal S}

\def \dual{\operatorname{dual}}

\def \HF{{\operatorname{HF}}}

\def \chara{{\operatorname{char}}}

\def \height{{\operatorname{ht}}}
\def \reg{{\operatorname{reg}}}
\def \depth{{\operatorname{depth}}}
\def \Gin{{\operatorname{Gin}}}
\def \grade{{\operatorname{grade}}}
\def \Ker{{\operatorname{Ker}}}
\def \pd{{\operatorname{pd}}}

\def \NN{\mathbb N}
\def \ZZ{\mathbb Z}

\def \S_d{\mathcal{M}(d)}

\def \S{\mathcal S}

\def\ini{\operatorname{\rm in}}

\begin{document}
\title{An intriguing ring structure on the set of $d$-forms}
\author{J\"urgen Herzog}
\address{Fachbereich Mathematik, Universit\"at Duisburg-Essen, Campus Essen, 45117
Essen, Germany}
\email{juergen.herzog@uni-essen.de}
\author{Leila Sharifan}
\address{Department of Mathematics, Hakim Sabzevari University, Sabzevar, Iran}
\email{leila-sharifan@aut.ac.ir}
\author{Matteo Varbaro}
\address{Dipartimento di Matematica,
Universit\`a degli Studi di Genova, Italy}
\email{varbaro@dima.unige.it}
\date{{\small \today}}
\maketitle

\begin{abstract}
The purpose of this note is to introduce a multiplication on the set of homogeneous polynomials of fixed degree $d$, in a way to provide a duality theory between monomial ideals of $K[x_1,\ldots ,x_d]$ generated in degrees $\leq n$ and block stable ideals (a class of ideals containing the Borel fixed ones) of $K[x_1,\ldots ,x_n]$ generated in degree $d$. As a byproduct we give a new proof of the characterization of Betti tables of ideals with linear resolution given by Murai.
\end{abstract}

\section*{Introduction}
Minimal free resolutions of modules over a polynomial ring are a classical and fascinating subject. Let $P=K[x_1,\ldots ,x_n]$ denote the polynomial ring equipped with the standard grading in $n$ variables over a field $K$. For a $\ZZ$-graded finitely generated $P$-module $M$, we consider its minimal graded free resolution:
\[\ldots\rightarrow \bigoplus_{j\in \ZZ}P(-j)^{\beta_{i,j}(M)}\rightarrow \ldots \rightarrow \bigoplus_{j\in \ZZ}P(-j)^{\beta_{0,j}(M)}\rightarrow M \rightarrow 0,\]
where $P(k)$ denotes the $P$-module $P$ supplied with the new grading $P(k)_i=P_{k+i}$. Hilbert's Syzygy theorem guarantees that the resolution above is finite: more precisely $\beta_{i,j}(M)=0$ whenever $i > n$. The natural numbers $\beta_{i,j}=\beta_{i,j}(M)$ are numerical invariants of $M$, and they are called the {\it graded Betti numbers} of $M$. The coarser invariants $\beta_i=\beta_i(M)=\sum_{j\in \ZZ}\beta_{i,j}$ are called the {\it (total) Betti numbers} of $M$. We will refer to the matrix $(\beta_{i,i+j})$ as the {\it Betti table} of $M$:
\[
\left(\begin{array}{cccccc}
\vdots & \vdots & \vdots & \cdots & \cdots & \vdots \\
                     \beta_{0,d} & \beta_{1,1+d} & \beta_{2,2+d} & \cdots & \cdots & \beta_{n,n+d}\\
                     \vdots & \vdots & \vdots & \cdots & \cdots & \vdots \\
                   \end{array}\right).
\]
 It is a classical problem to inquire on the behavior of Betti tables, especially when $M=P/I$ (equivalently $M=I$) for a graded ideal $I\subset P$. Recently the point of view is substantially changed: Boij and S\"oderberg in \cite{BS} suggested to look at the set of Betti tables of modules $M$ up to rational numbers. Eisenbud and Schreyer confirmed this intuition in \cite{ES}, giving birth to a new theory that demonstrated extremely powerful and is rapidly developing.

In some directions the original problem of determining the exact (not only up to rationals) possible values of the Betti numbers of ideals has however been solved: For example, Murai characterized the Betti tables of {\it ideals with linear resolution} (i.e. with only one nonzero row in the Betti table) in \cite[Proposition 3.8]{Mu}, and Crupi and Utano in \cite{CU} and the three authors of this paper in \cite{HSV} gave (different in nature) characterizations of the possible {\it extremal Betti numbers} (nonzero top left corners in a block of zeroes in the Betti table) that a graded ideal may achieve. The proof of Murai makes use of the Kalai's stretching of a monomial ideal and the Eagon-Reiner theorem. In this note we aim to give an alternative proof of his result, introducing a structure of $K$-algebra on the set of the degree $d$ polynomials in a suitable way to yield a good duality theory between strongly stable ideals of $K[x_1,\ldots ,x_d]$ generated in degrees $\leq n$ and strongly stable ideals of $K[x_1,\ldots ,x_n]$ generated in degree $d$. Such a duality extends to all monomial ideals of $K[x_1,\ldots ,x_d]$ generated in degrees $\leq n$, the counterpart being certain monomial ideals of $K[x_1,\ldots ,x_n]$ generated in degree $d$, which we will call {\it block stable ideals}. Let us remark that this construction is completely elementary.

\section{Terminology}
Throughout we denote by $\NN$ the set of the natural numbers $\{0,1,2,\ldots\}$ and by $n$ a positive natural number. We will essentially work with the polynomial rings
\[S=K[x_i:i\in \NN]\]
and
\[P=K[x_1,\ldots ,x_n],\]
where the $x_i$'s are variables over a field $K$. The reason why we consider a polynomial ring in infinite variables is that it is more natural to deal with it in Section \ref{section*}, when we will define the $*$-operation. However, for the applications of the theory to the graded Betti numbers, $P$ will be considered. The following notions will be introduced just relatively to $S$, also if we will use them also for $P$.

\vspace{3mm}

The ring $S$ is graded on $\NN$, namely $S=\bigoplus_{d\in \NN}S_d$ where
\[S_d = \langle x_{i_1}x_{i_2}\cdots x_{i_d} \ : \ i_1\leq i_2 \leq \ldots \leq i_d \mbox{ \ are natural numbers}\rangle.\]
Given a monomial $u\in S_d$, with $d\geq 1$, we set:
\begin{equation}\label{maximumvariable}
m(u)=\max \{e\in \NN \ : \ x_e \ \mbox{ divides } \ u\}.
\end{equation}
A {\it monomial space} $V\subset S$ is a $K$-vector subspace of $S$ which has a $K$-basis consisting of monomials of $S$. If $V\subset S_d$, we will refer to the {\it complementary monomial space $V^c$ of $V$} as the $K$-vector space generated by the monomials of $S_d$ which are not in $V$. Given a monomial space $V\subset S$ and two natural numbers $i,d$, such that $d\geq 1$, we set:
\begin{equation*}
w_{i,d}(V)=|\{u\mbox{ monomials in }V\cap S_d \ \mbox{ and } \ m(u)=i\}|.
\end{equation*}
Without taking in consideration the degrees,
\begin{equation*}
w_i(V)=|\{u\mbox{ monomials in }V \ \mbox{ and } \ m(u)=i\}|.
\end{equation*}
We order the variables of $S$ by the rule
\[x_i > x_j \ \iff \ i<j, \]
so that $x_0>x_1>x_2>\ldots$. On the monomials, unless we explicitly say differently, we use a degree lexicographical order w.r.t. the above ordering of the variables. 
A monomial space $V\subset S$ is called {\it stable} if for any monomial $u\in V$, then $(u/x_{m(u)})\cdot x_i\in V$ for all $i< m(u)$. It is called {\it strongly stable} if for any monomial $u\in V$ and for each $j\in \NN$ such that $x_j$ divides $u$, then $(u/x_j)\cdot x_i\in V$ for all $i< j$. Obviously a strongly stable monomial space is stable.

\vspace{3mm}

The remaining definitions of this section will be given for $P$, since we do not need them for $S$. A monomial space $V\subset P$ is called {\it lexsegment} if, for all $d\in \NN$, there exists a monomial $u\in P_d$ such that
\[V\cap P_d=\langle v\in P_d:v\geq u\rangle.\]
Clearly, a lexsegment monomial space is strongly stable. The celebrated theorem of Macaulay explains when a lexsegment monomial space is an ideal. We recall that given a natural number $a$ and a positive integer $d$, the $d$th {\it Macaulay representation} of $a$ is the unique writing:
\[a=\sum_{i=1}^d\binom{k(i)}{i} \ \ \ \mbox{such that } \ k(d)>k(d-1)>\ldots>k(1)\geq 0,\]
see \cite[Lemma 4.2.6]{BH}. Then:
\[a^{\langle d\rangle}=\sum_{i=1}^d\binom{k(i)+1}{i+1}.\]
A numerical sequence $(h_i)_{i\in\NN}$ is called {\it O-sequence} if $h_0=1$ and $h_{d+1}\leq h_d^{\langle d\rangle}$ for all $d\geq 1$. (The reader should be careful because the definition of $O$-sequence depends on the numbering: A vector $(m_1,\ldots ,m_n)$ will be a $O$-sequence if $m_1=1$ and and $m_{i+1}\leq m_i^{\langle i-1\rangle}$ for all $i\geq 2$). The theorem of Macaulay (for example see \cite[Theorem 4.2.10]{BH}) says that, given a numerical sequence $(h_i)_{i\in\NN}$, the following are equivalent:
\begin{itemize}
\item[(i)] $(h_i)_{i\in\NN}$ is an {\it O}-sequence with $h_1\leq n$.
\item[(ii)] There is a homogeneous ideal $I\subset P$ such that $(h_i)_{i\in\NN}$ is the Hilbert function of $P/I$.
\item[(iii)] The lexsegment monomial space $L\subset P$ such that $L\cap P_d$ consists in the biggest $\binom{n+d-1}{d}-h_d$ monomials, is an ideal.
\end{itemize}

We already defined the Betti numbers of a $\ZZ$-graded $P$-module $M$ in the introduction. For an integer $d$, the $P$-module $M$ is said to have a {\it $d$-linear resolution} if $\beta_{i,j}(M)=0$ for every $j\neq i+d$; equivalently, if $\beta_i(M)=\beta_{i,i+d}(M)$ for all $i$. Notice that if $M$ has $d$-linear resolution, then it is generated in degree $d$. The $P$-module $M$ is said {\it componentwise linear} if $M_{\langle d\rangle}$ has $d$-linear resolution for all $d\in\ZZ$, where $M_{\langle d\rangle}$ means the $P$-submodule of $M$ generated by the elements of degree $d$ of $M$. It is not difficult to show that if $M$ has a linear resolution, then it is componentwise linear.

We introduce the following numerical invariants of a $\ZZ$-graded finitely generated $P$-module $M$: For all $i=1,\ldots ,n+1$ and $d\in\ZZ$:
\begin{equation}\label{defmidgen}
m_{i,d}(M)=\sum_{k=0}^n(-1)^{k-i+1}\binom{k}{i-1}\beta_{k,k+d}(M).
\end{equation}
The following lemma shows that to know the $m_{i,d}(M)$'s is equivalent to know the Betti table of $M$.

\begin{lemma}\label{m_ib_i}
Let $M$ be a $\ZZ$-graded finitely generated $P$-module. Then:
\begin{equation}\label{eliker}
\beta_{i,i+d}(M)=\sum_{k=i}^{n+1}\binom{k-1}{i}m_{k,d}(M).
\end{equation}
\end{lemma}
\begin{proof}
Set $m_{k,d}=m_{k,d}(M)$ and $\beta_{i,j}=\beta_{i,j}(M)$. By the definition of the $m_{k,d}$'s we have the following identity in $\ZZ[t]$:
\[\sum_{k=1}^{n+1} m_{k,d}t^{k-1}=\sum_{i=0}^n\beta_{i,i+d}(t-1)^i.\]
Replacing $t$ by $s+1$, we get the identity of $\ZZ[s]$
\[\sum_{k=1}^{n+1}m_{k,d}(s+1)^{k-1}=\sum_{i=0}^n\beta_{i,i+d}s^i,\]
that implies the lemma.
\end{proof}

Let us define also the coarser invariants:
\begin{equation}\label{defmigen}
m_i(M)=\sum_{d\in\ZZ}m_{i,d}(M) \ \ \ \forall \ i=1,\ldots ,n+1.
\end{equation}
If $M=I$ is a homogeneous ideal of $P$, notice that $m_{i,d}=0$ if $i=n+1$ or $d<0$. We say that a monomial ideal $I\subset P$ is stable (strongly stable) (lexsegment) if the underlying monomial space is. By $G(I)$, we will denote the unique minimal set of monomial generators of $I$.
If $I$ is a stable monomial ideal, we have the following nice interpretation by the Eliahou-Kervaire formula \cite{EK} (see also \cite[Corollary 7.2.3]{HH2}):
\begin{eqnarray}\label{defmid}
m_{i,d}(I)=w_{i,d}(\langle G(I) \rangle)=|\{u\mbox{ monomials in }G(I)\cap P_d \ \mbox{ and } \ m(u)=i\}|\\
m_{i}(I)=w_{i}(\langle G(I) \rangle)=|\{u\mbox{ monomials in }G(I) \ \mbox{ and } \ m(u)=i\}|.\nonumber
\end{eqnarray}
From Lemma \ref{m_ib_i} and \eqref{defmid} follows that a stable ideal generated in degree $d$ has a $d$-linear resolution. Furthermore, if $I$ is a stable ideal, then $I_{\langle d\rangle}$ is stable for all natural numbers $d$. So any stable ideal is componentwise linear.

When $M=I$ is a stable monomial ideal we will consider \eqref{defmid} the definition of the $m_{i,d}$'s, and we will refer to \eqref{eliker} as the Eliahou-Kervaire formula.

\section{The $*$-operation on monomials and strongly stable ideals}\label{section*}

We are going to give a structure of associative commutative $K$-algebra to the $K$-vector space $S_d$, in the following way: Given two monomials $u$ and $v$ in $S_d$, we write them as $u=x_{i_1}x_{i_2}\cdots x_{i_d}$ with $i_1 \leq i_2 \leq \ldots \leq i_d$ and $v=x_{j_1}x_{j_2}\cdots x_{j_d}$ with $j_1 \leq j_2 \leq \ldots \leq j_d$. Then we define their product as
\begin{equation*}
u*v = x_{i_1+j_1}x_{i_2+j_2}\cdots x_{i_d+j_d}.
\end{equation*}
We can extend $*$ to the whole $S_d$ by $K$-linearity. Clearly, $*$ is associative and commutative. We will denote by $\S_d$ the $K$-vector space $S_d$ supplied with such an algebra structure. Actually $\S_d$ has a natural graded structure: In fact, we can write $\S_d = \oplus_{e\in \NN} (\S_d)_e$ where
\begin{equation*}
(\S_d)_e =\langle u\mbox{ monomial of } S_d \ \mbox{ and } \ m(u)=e\rangle.
\end{equation*}
Notice that $(\S_d)_0=\langle x_0^d\rangle \cong K$ and that $(\S_d)_e$ is a finite dimensional $K$-vector space. Therefore, $\S_d$ is actually a positively graded $K$-algebra. Moreover, if $u=x_0^{a_0}\cdots x_e^{a_e}\in \S_d$, with $a_e\neq 0$ and $e\geq 1$. Then
\[u=(x_0^{a_0}x_1^{a_1+\ldots + a_e})*(x_0^{a_0+a_1}x_1^{a_2+\ldots +a_e})*\ldots *(x_0^{a_0+\ldots +a_{e-1}}x_1^{a_e}),\]
so $\S_d$ is a standard graded $K$-algebra, that is $\S_d=K[(\S_d)_1]$. Particularly, $\S_d$ is Noetherian. Notice that $(\S_d)_1$ is a $K$-vector space of dimension $d$, namely:
\[(\S_d)_1=\langle x_0^{d-1}x_1,x_0^{d-2}x_1^2,\ldots ,x_1^d\rangle.\]
Actually, we are going to prove that $\S_d$ is a polynomial ring in $d$ variables over $K$.

\begin{prop}\label{structure}
The ring $\S_d$ is isomorphic, as a graded $K$-algebra, to the polynomial ring in $d$ variables over $K$.
\end{prop}
\begin{proof}
Let $K[y_1,\ldots ,y_d]$ be the polynomial ring in $d$ variables over $K$. Of course there is a graded surjective homomorphism of $K$-algebras $\phi$ from $K[y_1,\ldots ,y_d]$ to $\S_d$, by extending the rule:
\begin{equation}\label{identification}
\phi(y_i)=x_0^{i-1}x_1^{d+1-i}.
\end{equation}
In order to show that $\phi$ is an isomorphism, it suffices to exhibit an isomorphism of $K$-vector spaces between the graded components of $\S_d$ and $K[y_1,\ldots ,y_d]$. To this aim pick a monomial $u\in (\S_d)_e$:
\[u=x_0^{a_0}\cdots x_e^{a_e}, \ \ \ a_i\in \NN, \ a_e>0 \ \mbox{ and } \ \sum_{i=0}^ea_i=d.\]
To such a monomial we associate the monomial of $K[y_1,\ldots ,y_d]_e$
\[y_{a_0+1}y_{a_0+a_1+1}\cdots y_{a_0+\ldots +a_{e-1}+1}.\]
It is easy to see that the above application is one-to-one, so the proposition follows.
\end{proof}

\begin{remark}
For the sequel it is useful to familiarize with the map $\phi$. For instance, one can easily verify that:
\begin{equation}\label{imagephi}
\phi(y_1^{b_1}y_2^{b_2}\cdots y_d^{b_d})=x_{b_1}x_{b_1+b_2}\cdots x_{b_1+\ldots + b_d}.
\end{equation}
Proposition \ref{structure} guarantees that $\phi$ has an inverse, that we will denote by $\psi=\phi^{-1}:\S_d\rightarrow K[y_1,\ldots ,y_d]$. As one can show:
\begin{equation}\label{imagepsi}
\psi(x_0^{a_0}x_1^{a_1}\cdots x_e^{a_e})=y_{a_0+1}y_{a_0+a_1+1}\cdots y_{a_0+\ldots +a_{e-1}+1}.
\end{equation}
\end{remark}

Given a monomial space $V$ of course we have an isomorphism of $K$-vector spaces
\[V \cong \S_d/V^c.\]
However in general the above isomorphism does not yield a structure of $K$-algebra to $V$, because $V^c$ may be not an ideal of $\S_d$.
We are interested to characterize those monomials spaces $V\subset S_d$ such that $V^c$ is an ideal of $\S_d$. For what follows it is convenient to introduce the following definition.

\begin{definition}\label{def:sstable}
Let $V\subset S$ be a monomial space. We will call it {\it block stable} if for any $u=x_0^{a_0}\cdots x_e^{a_e}\in V$ and for any $i=1,\ldots ,e$, we have that
\[\frac{u}{x_i^{a_i}\cdots x_e^{a_e}}\cdot x_{i-1}^{a_i}\cdots x_{e-1}^{a_e}\in V.\]
\end{definition}

\begin{remark}\label{stablenotdouble}
Notice that a strongly stable monomial space is also stable and block stable. On the other side block stable monomial spaces might be not stable (it is enough to consider $\langle x_0^2, \ x_1^2\rangle$). There are also stable monomial spaces which are not block stable: Consider the monomial space:
\[V=\langle x_0^3, \ x_0^2x_1, \ x_0x_1^2, \ x_0x_1x_2, \ x_0x_1x_3\rangle \subset S_3.\]
It turns out that $V$ is stable, but not block stable, because
\[\frac{x_0x_1x_3}{x_1x_3}\cdot x_0x_2=x_0^2x_2\notin V.\]
Finally, the monomial space $\langle x_0^3, \ x_0^2x_1, \ x_0x_1^2, \ x_0x_1x_2\rangle \subset S_3$ is both stable and block stable, but is not strongly stable.
\end{remark}

\begin{lemma}\label{ideals}
Let $V\subset S_d$ be a monomial space. Then $V$ is block stable if and only if $V^c$ is an ideal of $\S_d$.
\end{lemma}
\begin{proof}
``Only if"-part. Consider a monomial $u\in V^c$. By contradiction there is an $i\in \{1,\ldots ,d-1\}$ such that
\[w=u*(x_0^ix_1^{d-i})\notin V^c.\]
If $u=x_{p_1}\cdots x_{p_d}$ with $ p_1 \leq \ldots \leq p_d$, then
\[w=x_{p_1}\cdots x_{p_i}\cdot x_{p_{i+1}+1}\cdots x_{p_d+1}.\]
Since $V$ is block stable and $w$ is a monomial of $V$, then
\[u=\frac{w}{x_{p_{i+1}+1}\cdots x_{p_d+1}}\cdot x_{p_{i+1}}\cdots x_{p_d}\in V,\]
a contradiction.

``If"-part. Pick $u=x_0^{a_0}\cdots x_e^{a_e} \in V$. By contradiction there is $i\in \{1,\ldots ,e\}$ such that
\[w=\frac{u}{x_i^{a_i}\cdots x_e^{a_e}}\cdot x_{i-1}^{a_i}\cdots x_{e-1}^{a_e}\notin V.\]
Since $V^c$ is an ideal of $\S_d$ and $w\in V^c$, we have
\[u=w*(x_0^{a_1+\ldots +a_{i-1}}x_1^{a_i+\ldots +a_e})\in V^c.\]
This contradicts the fact that we took $u\in V$.
\end{proof}

The following corollary, essentially, is why we introduced $\S_d$.

\begin{corollary}\label{bound}
Let $(w_i)_{i\in \NN}$ be a sequence of natural numbers. If there exists a strongly stable monomial space $V\subset S_d$ (actually it is enough that $V$ is block stable) such that $w_i(V)=w_i$ for any $i\in \NN$, then $(w_i)_{i\in \NN}$ is an {\it O}-sequence such that $w_1\leq d$.
\end{corollary}
\begin{proof}
That $w_0=1$ and $w_1\leq d$ is clear. By Lemma \ref{ideals} $V^c$ is an ideal of $\S_d$. So, Proposition \ref{structure} implies that $\S_d/V^c$ is a standard graded $K$-algebra. Clearly we have
\[\HF_{\S_d/V^c}(i)=w_i(V)=w_i \ \ \ \forall \ i\in\NN,\]
($\HF$ denotes the Hilbert function) so we get the conclusion by the theorem of Macaulay.
\end{proof}

The above corollary can be reversed. To this aim we need to understand the meaning of ``strongly stable" in $\S_d$. By Proposition \ref{structure} $\S_d\cong K[y_1,\ldots ,y_d]$, so we already have a notion of ``strongly stable" in $\S_d$. However, we want to describe it in terms of the multiplication $*$.

\begin{lemma}\label{sstableinM}
Let $W$ be a monomial space of $K[y_1,\ldots ,y_d]$. We recall the isomorphism $\phi : K[y_1,\ldots ,y_d]\rightarrow \S_d$ of \eqref{identification}. The following are equivalent:
\begin{itemize}
\item[{\em (i)}] $W$ is a strongly stable monomial space.
\item[{\em (ii)}] If $x_0^{a_0}\cdots x_e^{a_e}\in \phi(W)$ with $a_e>0$, then $x_0^{a_0}\cdots  x_i^{a_i-1}\cdot x_{i+1}^{a_{i+1}+1}\cdots x_e^{a_e}\in \phi(W)$ for all $i\in\{0,\ldots ,e-1\}$ such that $a_i>0$.
\end{itemize}
\end{lemma}
\begin{proof}
(i) $\implies$ (ii). If $u=x_0^{a_0}x_1^{a_1}\cdots x_e^{a_e}\in \phi(W)$ with $a_e>0$, then
\[\psi(u)=y_{a_0+1}y_{a_0+a_1+1}\cdots y_{a_0+\ldots +a_{e-1}+1}\in W,\]
see \eqref{imagepsi}. Since $W$ is strongly stable, then for all $i\in\{0,\ldots , e-1\}$:
\[w=y_{a_0+1}\cdots y_{a_0+\ldots +(a_i-1)+1}\cdot y_{a_0+\ldots +(a_i-1)+(a_{i+1}+1)+1}\cdots y_{a_0+\ldots +a_{e-1}+1}\in W.\]
Therefore, if $a_i>0$, we get $v=x_0^{a_0}\cdots x_i^{a_i-1}\cdot x_{i+1}^{a_{i+1}+1}\cdots x_e^{a_e}=\phi(w)$, so $v\in \phi(W)$.

(ii) $\implies$ (i). Let $w=y_1^{b_1}y_2^{b_2}\cdots y_d^{b_d}\in W$. Then, using \eqref{imagephi},
\[\phi(w)=x_{b_1}x_{b_1+b_{2}}\cdots x_{b_1+\ldots +b_d}\in \phi(W).\]
By contradiction there exist $p$ and $q$ in $\{1,\ldots ,d\}$ such that $b_p>0$, $q<p$ and
\[\frac{w}{y_p}\cdot y_q = y_1^{b_1}\cdots y_q^{b_q+1}\cdots y_p^{b_p-1}\cdots y_d^{b_d}\notin W.\]
Of course we can suppose that $q=p-1$, so we get a contradiction, because the assumptions yield:
\[\phi\left(\frac{w}{y_p}\cdot y_{p-1}\right)=x_{b_1}\cdots x_{b_1+\ldots +(b_{p-1}+1)}x_{b_1+\ldots +(b_{p-1}+1)+(b_{p}-1)}\cdots x_{b_1+\ldots +b_d}\in \phi(W).\]
\end{proof}

Thanks to Lemma \ref{sstableinM}, therefore, it will be clear what we mean for a monomial space of $\S_d$ being strongly stable.

\begin{prop}\label{sstableprop}
Let $V\subset S_d$ be a monomial space. The following are equivalent:
\begin{itemize}
\item[{\em (i)}] $V^c$ is a strongly stable monomial subspace of $\S_d$;
\item[{\em (ii)}] $V$ is a strongly stable monomial subspace of $S_d$.
\end{itemize}
\end{prop}
\begin{proof}
First we prove (i) $\implies$ (ii). Pick $u=x_0^{a_0}\cdots x_e^{a_e}\in V$. By contradiction, assume that there exists $i\in \{1,\ldots ,e\}$ such that $w=x_0^{a_0}\cdots x_{i-1}^{a_{i-1}+1}x_i^{a_i-1}\cdots x_e^{a_e}\notin V$. So $w\in V^c$, and since $V^c$ is a strongly stable monomial ideal of $\S_d$, by Lemma \ref{sstableinM} we get $u\in V^c$, which is a contradiction.

\vskip 1mm

(ii) $\implies$ (i). By Lemma \ref{ideals} we have that $V^c$ is an ideal of $\S_d$. Consider $u=x_0^{a_0}\cdots x_e^{a_e}\in V^c$ with $a_e>0$ and $i\in \{0,\ldots ,e-1\}$. If $w=x_0^{a_0}\cdots x_i^{a_i-1}\cdot x_{i+1}^{a_{i+1}+1}\cdots x_e^{a_e}$ were not in $V^c$, then $u$ would be in $V$ because $V$ is a strongly stable monomial space. Thus $V^c$ has to be strongly stable once again using Lemma \ref{sstableinM}.
\end{proof}

\begin{thm}\label{characterization}
Let $(w_i)_{i\in \NN}$ be a sequence of natural numbers. Then the following are equivalent:
\begin{itemize}
\item[{\em (i)}] There exists a strongly stable monomial space $V\subset S_d$ such that $w_i(V)=w_i$ \ for any $i\in \NN$.
\item[{\em (ii)}] There exists a block stable monomial space $V\subset S_d$ such that $w_i(V)=w_i$ \ for any $i\in \NN$.
\item[{\em (iii)}] $(w_i)_{i\in \NN}$ is an {\it O}-sequence such that $w_1\leq d$.
\end{itemize}
\end{thm}
\begin{proof}
(i) $\implies$ (ii) is obvious and (ii) $\implies$ (iii) is Corollary \ref{bound}. So (iii) $\implies$ (i) is the only thing we still have to prove. If the sequence $(w_i)_{i\in \NN}$ satisfies the conditions of (iii), then the theorem of Macaulay guarantees that there exists a lexsegment ideal $J\subset K[y_1,\ldots ,y_d]$ such that
\begin{equation*}
\HF_{K[y_1,\ldots ,y_d]/J}(i)= w_i \ \ \ \forall \ i\in \NN
\end{equation*}
Being a lexsegment ideal, $J$ is strongly stable. So $\phi(J)^c$ is a strongly stable monomial subspace of $S_d$ by Proposition \ref{sstableprop}. Clearly we have:
\[m_i(\phi(J)^c)=\HF_{K[y_1,\ldots ,y_d]/J}(i)=w_i \ \ \ \forall \ i\in \NN,\]
thus we conclude.
\end{proof}

\bigskip

\bigskip

\begin{discussion}

Theorem~\ref{mainlinear} implies \cite[Proposition 3.8]{Mu}. Let us briefly discuss the proof of Murai, comparing it with ours.

Let $u=x_{i_1}x_{i_2}\cdots x_{i_d}$ be a monomial with $i_1\leq i_2\leq \ldots\leq i_d$. Following Kalai, the {\em stretched} monomial  arising from $u$ is
\[
u^{\sigma}=x_{i_1}x_{i_2+1}\cdots x_{i_d+(d-1)}.
\]
Notice that $u^\sigma$ is a squarefree monomial.
The {\it compress operator} $\tau$ is the inverse to $\sigma$. If $v=x_{j_1}x_{j_2}\cdots x_{j_d}$ is a squarefree monomial, we define the {\em compressed} monomial arising from $v$ to be
\[
v^{\tau}=x_{j_1}x_{j_2-1}\cdots x_{j_d-(d-1)}.
\]
Let $I\subset K[x_1,\ldots ,x_n]$  be a strongly stable ideal generated in degree $d$ with $G(I)=\{u_1,\ldots,u_r\}$. We set
\[
I^\sigma=(u_1^\sigma, u_2^\sigma,\ldots,u_r^\sigma)\subset K[x_1,\ldots,x_{n+m-1}].
\]
As shown in \cite[Lemma 11.2.5]{HH2}, one has  that $I^\sigma$ is a squarefree strongly stable ideal.  Recall that a squarefree monomial ideal $J$ is called {\em squarefree strongly stable}, if for all squarefree generators  $u$ of $I$ and all  $i<j$ for which  $x_j$ divides $u$ and $x_i$ does not divides $u$, one has that $(u/x_j)\cdot x_i\in J$.
Denoting by $^\vee$ the {\it Alexander dual} of a squarefree monomial ideal, given a strongly stable ideal $I$ we set
\[
I^{\dual}=((I^\sigma)^\vee)^\tau,
\]
where for a squarefree monomial ideal $J$ with $G(J)=\{u_1,\ldots,u_m\}$ we set $J^\tau=(u_1^\tau,\ldots,u_m^\tau)$. Murai showed his result using a formula relating the Betti numbers of a squarefree monomial ideal with linear resolution and the $h$-vector of the quotient by its Alexander dual, that is Cohen-Macaulay by the Eagon-Reiner theorem. 

Starting with a strongly stable monomial ideal is necessary, otherwise the stretching operator changes the Betti numbers. However, one can show that on strongly stable ideals this duality actually coincides with the one discussed in this note: If $J'\subset K[x_1,\cdots ,x_{n}]$ is a strongly stable ideal  generated in degree $d$ and $J\subset S$ is the ideal $J'S$ under the transformation $x_i\mapsto x_{i+1}$, then
\[\psi(\langle G(J)^c\rangle)=J'^{\dual}\]
up to degree $n$ ($J'^{dual}$ has not minimal generators of degree bigger than $n$).
To show this, it is enough to notice that $J'^{\dual}\subset \psi(\langle G(J)^c\rangle)$ because the graded rings $K[x_1,\cdots ,x_d]/ J'^{\dual}$ and
$K[x_1,\cdots ,x_d]/\langle G(J)^c\rangle$ share the same Hilbert function up to $n$. 

\end{discussion}

\bigskip 

\bigskip

Actually, a careful reading of the proof of Theorem \ref{characterization} shows that, given a $O$-sequence, we can give explicitly a strongly stable monomial subspace $V\subset S_d$ such that $w_i(V)=w_i$ for any $i\in\NN$. The reason is that to any Hilbert function is associated a unique lexsegment ideal: Let $(w_i)_{i\in \NN}$ be a sequence of natural numbers. For any $i\in \NN$, set
\[ V_i = \{\mbox{biggest $w_i$ monomials $u\in S_d$ such that $m(u)=i$}\}.\]
Then we call $V=\langle \cup_{i\in \NN} V_i\rangle \subset S_d$ the {\it piecewise lexsegment} monomial space (of type $(d,(w_i)_{\NN})$). The proof of Theorem \ref{characterization} yields:
\begin{corollary}
The piecewise lexsegment of type $(d,(w_i)_{\NN})$ is strongly stable if and only if $(w_i)_{\NN}$ is an $O$-sequence such that $w_1\leq d$.
\end{corollary}

\begin{remark}
The notion of piecewise lexsegment was successfully used in \cite{HSV} to characterize the possible extremal Betti numbers of a homogeneous ideal. We wish to point out that, even if \cite[Theorem 3.7]{HSV} is stated in characteristic $0$, actually the same conclusion holds true in any characteristic, by exploiting a construction given by Caviglia and Sbarra in \cite{CS} (see Proposition 2.2(vi) of that paper). 
\end{remark}

Notice that the established interaction between $S_d$ and $K[y_1,\ldots ,y_d]$ can be also formulated between
\[K[x_0,\ldots ,x_m] \ \mbox{ \ and \ } \ K[y_1,\ldots ,y_d]/(y_1,\ldots ,y_d)^{m+1} \ \ \ \forall \ m\geq 1.\]
Therefore, an interesting corollary of Proposition \ref{sstableprop} is the following.

\begin{corollary}\label{Matteo correspondence}
Let us define the sets
\[A=\{\mbox{strongly stable monomial ideals of $K[x_0,\ldots ,x_m]$ generated in degree $d$}\}\]
and
\[B=\{ \mbox{strongly stable monomial ideals of  $K[y_1,\ldots ,y_d]$ with height $d$ and generated in  degrees $\leq m+1$}\}.\]
Then the assignation $V\mapsto \psi(V^c)$ establishes a 1-1 correspondence between $A$ and $B$.
\end{corollary}
\begin{proof}
Notice that if $I\subset K[y_1,\ldots ,y_d]$ is of height $d$, then $(y_1,\ldots ,y_d)^k\subset I$ for all $k\geq \reg(I)$. Since $I$ is generated in degrees $\leq m+1$ and componentwise linear, we have $\reg(I)\leq m+1$, so we are done by what said before the corollary.
\end{proof}
%

It is worth to rest a bit on the properties of block stable ideals, since by Lemma \ref{ideals} they seem to arise naturally by studying strongly stable ideals. Let us consider the Borel subgroup of $\GL_{\infty}(K)$ consisting of $\infty \times \infty$ upper diagonal matrices with entries in $K$ and 1's on the diagonal. In characteristic $0$ Borel fixed (w.r.t. the obvious action) monomial spaces are strongly stable, so in particular block stable. However in positive characteristic the situation is quite different, for example the space $\langle x_0^2,x_1^2\rangle$ is Borel fixed in characteristic $2$ but not strongly stable. 

\begin{prop}
Regardless to $\chara(K)$, a Borel fixed monomial space is block stable.
\end{prop}
\begin{proof}
Let $V\subset S$ be a Borel fixed vector space. If $u=x_0^{a_0}\cdots x_e^{a_e}\in V$, then for any $i=1,\ldots ,e$ we have that $(u/x_i^{a_i})\cdot x_{i-1}^{a_{i}}\in V$ since $\binom{a_i}{a_i}=1$ is different from $0$ modulo $\chara(K)$, whatever the latter is (see \cite[Theorem 15.23]{Ei}). Recursively one gets
\[\frac{u}{x_i^{a_i}\cdots x_e^{a_e}}\cdot x_{i-1}^{a_i}\cdots x_{e-1}^{a_e}\in V.\]
\end{proof}

One might be induced to look for an analog of the Eliahou-Kervaire formula for block stable ideals. Such a formula, however, would be not purely combinatorial, in the sense that the graded Betti numbers of block stable ideals depend on the characteristic of the field $K$: In fact even the Betti numbers of a Borel fixed ideal depend on the characteristic, as recently shown (indeed while they were at MSRI for the 2012 ``Commutative Algebra" program) by Caviglia and Kummini in \cite[Theorem 3.2]{CM}, solving negatively a conjecture of Pardue. Their method gives rise to a Borel fixed ideal generated in many degrees. However Caviglia pointed out to us that we can even get a Borel fixed ideal generated in a single degree as follows:

\begin{example}
There is an ideal $I\subset R=\ZZ[x_1,\ldots ,x_6]$ generated in a single degree 2726 such that it is Borel fixed in characteristic 2 but its Betti numbers depend on the characteristic.

{\it Proof}: Let $J\subset R$ the Borel fixed ideal (in characteristic 2) of \cite[Example 3.7]{CM}. If $d=2729$, we have that $\beta_{2,d}(J(R\otimes_{\ZZ}K))$ is $0$ or not according to $\chara(K)$ being different or equal to $2$. By computing the Betti numbers in terms of Koszul homology w.r.t. $(x_1,\ldots ,x_6)$, it is clear that
\[\beta_{2,d}(J(R\otimes_{\ZZ}K))=\beta_{2,d}((J_{d-2}+J_{d-3})(R\otimes_{\ZZ}K)).\]
However the minimal generator of maximal degree of $J$ has degree 2568, that is less than $d-3$. So $(J_{d-3})=(J_{d-2}+J_{d-3})$. In particular $I=(J_{d-3})$ is a Borel fixed ideal (in characteristic 2) generated in degree 2726 whose Betti numbers are sensible to the characteristic.

\end{example}

\subsection{The possible Betti numbers of an ideal with linear resolution}

In this subsection we will see how Theorem \ref{characterization} yields a characterization of the Betti tables with just one row. Such an issue, in fact, is equivalent to characterize the possible graded Betti numbers of a strongly stable monomial ideal of $P$ generated in one degree. Actually, more generally, to characterize the possible Betti tables of a componentwise linear ideal of $P$ is equivalent to characterize the possible Betti tables of a strongly stable monomial ideal of $P$. In fact, in characteristic $0$ this is true because the generic initial ideal of any ideal $I$ is strongly stable \cite[Theorem 15.23]{Ei}. Moreover, if $I$ is componentwise linear and the term order is degree reverse lexicographic, then the graded Betti numbers of $I$ are the same of those of $\Gin(I)$ by a result of Aramova, Herzog and Hibi in \cite{AHH}. In positive characteristic it is still true that for a degree reverse lexicographic order the graded Betti numbers of $I$ are the same of those of $\Gin(I)$, provided that $I$ is componentwise linear. But in this case $\Gin(I)$ might be not strongly stable. However, it is known that, at least for componentwise linear ideals, it is stable \cite[Lemma 1.4]{CHH}. The graded Betti numbers of a stable ideal do not depend from the characteristic, because the Elihaou-Kervaire formula \eqref{eliker}. So to compute the graded Betti numbers of $\Gin(I)$ we can consider it in characteristic $0$. Let us call $J$ the ideal $\Gin(I)$ viewed in characteristic $0$. The ideal $J$, being stable, is componentwise linear, so we are done by what said above. Summarizing, we showed:

\begin{prop}\label{bc=bs}
The following sets coincide:
\begin{enumerate}
\item[{\em (1)}] $\{\mbox{Betti tables $(\beta_{i,j}(I))$ where $I\subset P$ is componentwise linear}\}$;
\item[{\em (2)}] $\{\mbox{Betti tables $(\beta_{i,j}(I))$ where $I\subset P$ is strongly stable}\}$;
\end{enumerate}
\end{prop}

So, we get the following:

\begin{thm}\label{mainlinear}
Let $m_1,\ldots ,m_n$ be a sequence of natural numbers. Then the following are equivalent:
\begin{enumerate}
\item[{\em (1)}] There exists a homogeneous ideal $I\subset P$ with $d$-linear resolution such that $m_k(I)=m_k$ for all $k=1,\ldots ,n$;
\item[{\em (2)}] There exists a strongly stable monomial ideal $I\subset P$ generated in degree $d$ such that $m_k(I)=m_k$ for all $k=1,\ldots ,n$;
\item[{\em (3)}] $(m_1,\ldots ,m_n)$ is an $O$-sequence such that $m_2\leq d$, that is:
\begin{enumerate}
\item[{\em (a)}] $m_1=1$;
\item[{\em (b)}] $m_2\leq d$;
\item[{\em (c)}] $m_{i+1}\leq m_{i}^{\langle i-1 \rangle}$ for any $i=2,\ldots ,n-1$.
\end{enumerate}
\end{enumerate}
\end{thm}
\begin{proof}
By virtue of Proposition \ref{bc=bs}, (1) $\iff$ (2). Moreover, if $I$ is strongly stable, then $m_i(I)=w_i(\langle G(I)\rangle )$ for all $i=1,\ldots ,n$, see \eqref{defmid}. Since the monomial space $\langle G(I)\rangle$ is strongly stable, Theorem \ref{characterization} yields the equivalence (2) $\iff$ (3).
\end{proof}

\begin{example}
Let us see an example: Theorem \ref{mainlinear} assures that we will never find a homogeneous ideal $I\subset R=K[x_1,x_2,x_3,x_4]$ with minimal free resolution:
\[0\longrightarrow R(-6)^6\longrightarrow R(-5)^{22}\longrightarrow R(-4)^{29}\longrightarrow R(-3)^{14}\longrightarrow I\longrightarrow 0.\]
In fact $I$, using \eqref{defmidgen}, should satisfy $m_1(I)=1$, $m_2(I)=3$, $m_3(I)=4$ and $m_4(I)=6$. This is not an $O$-sequence, thus the existence of $I$ would contradict Theorem \ref{mainlinear}.
\end{example}



\begin{thebibliography}{MPT}
\bibitem[AHH]{AHH} A. Aramova, J. Herzog, T. Hibi, {\it Ideals with stable Betti numbers}, Adv. Math. 152 (2000), no. 1, pp. 72--77.
\bibitem[BCP]{BCP} D. Bayer, H. Charalambous, S. Popescu, {\it Extremal Betti numbers and applications to monomial ideals}, J. Algebra 221 (1999), no. 2, pp. 497--512.
\bibitem[BS]{BS} M. Boij, J. S\"oderberg, {\it Graded Betti numbers of Cohen-Macaulay modules and the multiplicity conjecture}, J. Lond. Math. Soc. 78 (2008), pp. 85--106.
\bibitem[BH]{BH} W. Bruns, J. Herzog, {\it Cohen-Macaulay rings}, Cambridge Studies in Advanced Mathematics 39, 1993.
\bibitem[CM]{CM} G. Caviglia, M. Kummini, {\it Betti tables of $p$-Borel fixed ideals}, available at at \url{arXiv:1212.2201v1}, 2012.
\bibitem[CS]{CS} G. Caviglia, E. Sbarra, {\it Zero-generic initial ideals}, available at at \url{arXiv:1303.5373v1}, 2013.
\bibitem[Co]{cocoa} CoCoA Team, {\it CoCoA: a system for doing Computations in Commutative Algebra}, available at \url{http://cocoa.dima.unige.it}.
\bibitem[CHH]{CHH} A. Conca, J. Herzog, T. Hibi, {\it Rigid resolutions and big Betti numbers}, Comment. Math. Helv. 79 (2004), no. 4, pp. 826--839.
\bibitem[CU]{CU} M. Crupi, R. Utano, {\it Extremal Betti numbers of graded ideals}, Result. Math. 43 (2003), pp. 235-244.
\bibitem[ER]{ER} J.A. Eagon, V. Reiner, {\it Resolutions of Stanley Reisner ideals and  Alexander duality}, J. Pure and Appl. Algebra 130 (1988), no. 3, pp. 265--275.
\bibitem[Ei]{Ei} D. Eisenbud, {\it Commutative Algebra with a View Toward Algebraic Geometry}, Graduate Text in Mathematics 150, 1995.
\bibitem[EG]{EG} D. Eisenbud, S. Goto, {\it Linear Free Resolutions and Minimal Multiplicity}, J. Algebra 88 (1984), pp. 89--133.
\bibitem[ES]{ES} D. Eisenbud, F.-O. Schreyer, {\it Betti Numbers of Graded Modules and Cohomology of Vector Bundles}, J. Amer. Math. Soc. 22 (2009), pp. 859--888.
\bibitem[EK]{EK} S. Eliahou, M. Kervaire, {\it Minimal resolutions of some monomial ideals}, J. Algebra 129 (1990), no. 1, pp. 1--25.
\bibitem[HH1]{HH} J. Herzog, T. Hibi, {\it Componentwise linear ideals}, Nagoya Mathematical Journal 153, (1999), pp. 141--153.
\bibitem[HH2]{HH2} J. Herzog, T. Hibi, {\it Monomial ideals}, Graduate Texts in Mathematics 260, 2011.
\bibitem[HSV]{HSV} J. Herzog, L. Sharifan, M. Varbaro, {\it The possible extremal Betti numbers of a homogeneous ideal}, to appear in Proc. Amer. Math. Soc..
\bibitem[HS]{HS} J. Herzog, H. Srinivasan, {\it Bounds for multiplicities}, Trans. Amer. Math. Soc. 350, (1998), no. 7 pp. 2879--2902.
\bibitem[Mu]{Mu} S. Murai, {\it Hilbert functions of $d$-regular ideals}, J. Algebra 317 (2007), no. 2, pp. 658-690.
\bibitem[NR]{NR} U. Nagel, T. R\"omer, {\it Criteria for componentwise linearity}, available at \url{arXiv:1108.3921v2}, 2011.
\bibitem[Te]{Te} N. Terai, {\it Alexander duality theorem and Stanley-Reisner rings. Free resolutions of coordinate rings of projective varieties and related topics}, S\= urikaisekikenky\= usho K\= oki\= uroku No. 1078 (1999), pp. 174--184.
\end{thebibliography}
\end{document}